\documentclass[a4paper]{amsart}
\usepackage{graphicx}
\usepackage[bookmarksnumbered,colorlinks,plainpages]{hyperref}
\usepackage{extarrows}
\usepackage{float}
\usepackage{amsthm}
\usepackage[all,2cell]{xy}
\usepackage{amsfonts}
\usepackage{amsmath}
\usepackage{amssymb}
\usepackage{xypic,leqno,amscd,amssymb,pstricks,latexsym,amsbsy,xypic,mathrsfs,verbatim}
\usepackage{multirow}
\usepackage{booktabs}
\usepackage{makecell}
\usepackage{cases}
\usepackage{appendix}
\usepackage{graphicx,colortbl}
\usepackage{tikz-cd}
\usepackage{etex}
\usepackage{mathtools}
\usepackage[bookmarksnumbered,colorlinks,plainpages]{hyperref}
\hypersetup{
	colorlinks=true,
	citecolor=blue
}

\newcommand{\X}{\mathbb{X}}

\newcommand{\Z}{\mathbb{Z}}

\newcommand{\cut}{\ar@{-}@[|(5)]}

\newcommand{\Hom}{\operatorname{Hom}\nolimits}
\newcommand{\coh}{\operatorname{coh}\nolimits}
\newcommand{\End}{\operatorname{End}\nolimits}
\newcommand{\Ext}{\operatorname{Ext}\nolimits}
\newcommand{\rad}{\operatorname{rad}\nolimits}
\newcommand{\soc}{\operatorname{soc}\nolimits}

\newcommand{\bo}{\operatorname{b}\nolimits}

\newcommand{\RHom}{\mathbf{R}\strut\kern-.2em\operatorname{Hom}\nolimits}

\newcommand{\itrig}{\operatorname{Ind}\nolimits\tau\text{-}\operatorname{rigid}}

\DeclareMathOperator{\Ind}{{Ind}}

\DeclareMathOperator{\moduleCategory}{{mod}} \renewcommand{\mod}{\moduleCategory}
\newcommand{\DDD}{{D}}

\theoremstyle{plain}
\newtheorem{theorem}{Theorem}[section]
\newtheorem{lemma}[theorem]{Lemma}

\newtheorem{proposition}[theorem]{Proposition}
\theoremstyle{definition}
\newtheorem{definition}[theorem]{Definition}
\newtheorem{example}[theorem]{Example}
\newtheorem{remark}[theorem]{Remark}

\numberwithin{equation}{section}
\newtheorem{mainthm}{Theorem}

\tikzset{every picture/.style={line width=0.75pt}} 

\begin{document}

	\title[Representation type of higher preprojective algebras]{On the representation type of higher preprojective algebras of type $A$}

	\author[W. Weng] {Weikang Weng}
 
	\makeatletter \@namedef{subjclassname@2020}{\textup{2020} Mathematics Subject Classification} \makeatother
	

	\begin{abstract} 
	In this paper, we establish a finite--tame--wild trichotomy for the representation type of 
	$(n+1)$-preprojective algebras of  type $A_m$.
		We also provide  a complete classification of such algebras with respect to $\tau$-tilting finiteness. Moreover, for $n\ge 2$, we give a precise characterization of when each indecomposable module is $\tau$-rigid.
		 
	\end{abstract}
	
	\maketitle
	
	\section{Introduction}	
	Preprojective algebras were first introduced by Gelfand–Ponomarev \cite{GP}. Since then, they have become important objects in
	the representation theory of algebras and have found applications in various areas of mathematics. For instance, they play a central role in Lusztig’s Lagrangian construction of semicanonical bases \cite{Lu1,Lu2}, as well as in the theory of cluster algebras \cite{GLS}. Furthermore, preprojective algebras of Dynkin type $A$ are closely related to
	submodule categories \cite{RZ}.

	Higher preprojective algebras were developed by Iyama and Oppermann as higher-dimensional
	analogues of classical preprojective algebras in the framework of
	higher Auslander--Reiten theory \cite{IO1,IO2}. For an algebra $\Lambda$ of global
	dimension at most $n$, the $(n+1)$-preprojective algebra is defined as the  tensor algebra
	\[
	\Pi_{n+1}(\Lambda)
	:=
	T_{\Lambda}\Ext_{\Lambda}^{n}(D\Lambda,\Lambda).
	\]
	If $\Lambda$ is $n$-representation-finite, then
	$\Pi_{n+1}(\Lambda)$ is self-injective, and its
	stable module category is closely related to the corresponding higher
	Amiot cluster category \cite{IO1,IO2}. For an algebra $\Lambda$ of global dimension at most two,
	the $3$-preprojective algebra $\Pi_3(\Lambda)$  is isomorphic to a
	Jacobian algebra associated with a quiver with potential. At the differential graded (=dg) level, the derived
	$3$-preprojective algebra is quasi-isomorphic to the corresponding
	Ginzburg dg algebra, providing a link with generalized cluster
	categories and Calabi--Yau completions
	\cite{HI, KV}.
	
	 Let ${\bf k}$ be an algebraically closed field. Let $A_m$ be the linearly oriented quiver $$1\longrightarrow 2 \longrightarrow \dots \longrightarrow m,$$ and $A_m^1={\bf k}A_m$ be the path algebra of $A_m$. Recursively, let $A_m^{n+1}$ be the
	$n$-Auslander algebra of $A_m^n$. In this paper we study the family
	\[
	\Pi^{(n,m)}:=\Pi_{n+1}(A_m^n),
	\]
	which we call the \emph{higher preprojective algebra of type $A_m$}.
	
	When $n=1$, it is well-known that the algebra $\Pi^{(1,m)}$ coincides with the classical preprojective algebra of  type $A_m$.  It is representation-finite precisely when
	$m\leq4$; 
	$\Pi^{(1,5)}$ is tame, whereas $\Pi^{(1,m)}$ is wild for $m\geq6$; see
	\cite[Proposition~3.3]{GLS} and the references therein.
	The main purpose of this paper is to determine the representation type of
	$\Pi^{(n,m)}$ for  $n\ge 2$ and $m\ge 1$, and to compute the number of isomorphism classes of indecomposable
	modules in the representation-finite cases. Together with the known results for the case $n=1$, 
	we have the following theorem.

	\begin{mainthm} \label{main thm}
		Let $m,n\geq1$.  The representation type of ${\Pi}^{(n,m)}$ is given by the
		finite--tame--wild trichotomy as follows.
		\begin{equation*}
			\renewcommand{\arraystretch}{1.15}
			\begin{array}{c|c|c|c}
				\toprule
				n & \text{representation-finite} & \text{tame} & \text{wild} \\
				\midrule
				1 & m \leq 4 & m = 5 & m \geq 6 \\
				2 & m \leq 3 & m = 4 & m \geq 5 \\
				n \geq 3 & m \leq 3 & \text{none} & m \geq 4 \\
				\bottomrule
			\end{array}
		\end{equation*}
			In the representation-finite cases, we have
			\begin{equation} \label{numble}
			|\Ind {\Pi}^{(n,m)}|=
			\begin{cases}
				1, & m=1, \\[3pt]
				2(n+1), & m=2, \\[3pt]
				\dfrac{(n+1)(n+2)(n+3)}{2}, & m=3, \\[8pt]
				40, & (n,m)=(1,4),
			\end{cases}
			\end{equation}
			where $|\Ind {\Pi}^{(n,m)}|$ denotes the number of isomorphism classes of indecomposable ${\Pi}^{(n,m)}$-modules.
	\end{mainthm}
	
	Recently, the representation-finite higher Auslander algebras $A_m^n$ were classified in \cite{Li}. 
		Since $A_m^n$ is a
	 quotient algebra of ${\Pi}^{(n,m)}$ by Theorem \ref{Preprojective Auslander cut},  the representation-finiteness of ${\Pi}^{(n,m)}$ implies that of $A_m^n$ for $n\geq 2$ and $m\leq 3$. This recovers the representation-finiteness assertions in \cite[Propositions~4.4 and~4.5]{Li}. 	
	On the other hand, if $A_m^n$ is representation-infinite, then so is
	${\Pi}^{(n,m)}$. Thus, for
	$n\geq 2$ and $m\geq 5$, or $n\geq 3$ and $m\geq 4$, $A_m^n$  is representation-infinite
	 implies that so is ${\Pi}^{(n,m)}$; see \cite[Theorem~4.8]{Li}. However, this argument does not
	cover the exceptional case $(n,m)=(2,4)$, since $A_4^2$ is
	representation-finite. Moreover, it does not distinguish between tame and
	wild representation type. Theorem~\ref{main thm} completes this picture by
	providing a finite--tame--wild trichotomy.

	An important problem in $\tau$-tilting theory is to classify $\tau$-tilting finite algebras, namely, finite-dimensional algebras having only finitely many isomorphism classes of indecomposable $\tau$-rigid modules,  or equivalently, basic support $\tau$-tilting
	modules; see \cite{DIJ}. For this, we determine precisely when ${\Pi}^{(n,m)}$
	is $\tau$-tilting finite.

		\begin{mainthm} \label{main thm 2} Let $m,n\geq1$. 
				Then  ${\Pi}^{(n,m)}$ is $\tau$-tilting finite if and only if
				$n=1$ or $m\le 3$.
			\end{mainthm}
		
		The higher dimensional case behaves differently from the classical one.
		For $n=1$,  $\Pi^{(1,m)}$ is $\tau$-tilting finite for
		all $m\ge 1$, but it is representation-finite if and only if $m\le4$. 	The number of indecomposable 
		$\tau$-rigid $\Pi^{(1,m)}$-modules does not always equal $|\Ind {\Pi}^{(1,m)}|$.
		By contrast, for $n\ge2$, the relevant
		conditions coincide, as shown in the following theorem.
		
		\begin{mainthm} \label{main thm 3} Let $n\geq 2$. Then the following conditions are equivalent:
		\begin{enumerate}
			\item  $m\le 3$;
			\item 	$\Pi^{(n,m)}$ is representation-finite;
			\item  $\Pi^{(n,m)}$ is $\tau$-tilting finite;
			\item  each indecomposable  $\Pi^{(n,m)}$-module is $\tau$-rigid.
		\end{enumerate}
	\end{mainthm}
	
	Throughout this paper, let ${\bf k}$ be an algebraically closed field  and $D(-):=\Hom_{\mathbf{k}}(-,\mathbf{k})$. Fix integers $m,n\ge 1$.   By an algebra we mean a basic and finite dimensional $\mathbf{k}$-algebra, and by a module we mean a finite dimensional right module.
	For an algebra $\Lambda$, we denote by $\mod \Lambda$ the category of  $\Lambda$-modules, by $\underline{\mod}\, \Lambda$ the stable category, by $\DDD^{\bo}(\Lambda)$ the  bounded derived category of $\mod \Lambda$.

	\section{Preliminaries}
	In this section, we collect some results needed later.

	\subsection{Higher preprojective algebras of type $A_m$}	
	An algebra $\Lambda$ is called \emph{$n$-representation-finite} if $\Lambda$ has global dimension  at most $n$ and it admits an $n$-cluster tilting module $M\in \mod \Lambda$. Let $\Lambda$ be an $n$-representation-finite algebra. 
	The \emph{$(n+1)$-preprojective algebra of $\Lambda$}  is
	\begin{equation*}
		\Pi={\Pi}_{n+1}(\Lambda):=T_{\Lambda}\Ext_{\Lambda}^n(D\Lambda,\Lambda),
	\end{equation*}	
	that is,  the tensor algebra of the $\Lambda$-$\Lambda$ bimodule $\Ext_{\Lambda}^n(D\Lambda,\Lambda)$ over $\Lambda$. Then $\Pi$ is finite dimensional and self-injective by \cite[Corollary 3.4]{IO2}. Hence its stable module category $\underline{\mod} \, \Pi$ is a triangulated category.  As a $\Lambda$-module,  $\Pi$ is the unique basic $n$-cluster tilting module \cite[Theorem 2.21]{IO2} and we
	write
	\[
	\Gamma=\underline{\End}_{\Lambda}(\Pi)
	\]
	for its stable $n$-Auslander algebra. 
	It  follows from \cite[Theorem 4.15]{IO2} that there is a  triangle equivalence
	\begin{equation}\label{IO-equivalence}
		\underline{\mod}\,\Pi
		\simeq
		\mathcal C_{\Gamma}^{n+1}		
	\end{equation}
	between the stable module category of $\Pi$ and the $(n+1)$-Amiot cluster category of $\Gamma$. 
	In particular, if $\Gamma$ is derived equivalent to a hereditary algebra $A$, then by \cite[Theorem 7.1]{Ke} there is a triangle equivalence
	\begin{equation}\label{eq:dynkin-orbit}
		\underline{\mod}\,\Pi
		\simeq
		\DDD^{\bo}(A)/\tau^{-1}[n],				
	\end{equation}
	where $\tau^{-1}$ denotes the inverse of Auslander–Reiten translation and $[n]$ denotes the $n$-th suspension in $\DDD^{\bo}(A)$.
	
	Let $A_m$ be the linearly oriented quiver $$1\longrightarrow 2 \longrightarrow \dots \longrightarrow m$$ and $A_m^1={\bf k}A_m$ be the path algebra of $A_m$. Recursively, let $A_m^{n+1}$ be the
	$n$-Auslander algebra of $A_m^n$, that is, $A_m^{n+1}=\End_{A_m^{n}}(M^n_m)$, where $M^n_m$ is the unique $n$-cluster tilting $A_m^{n}$-module. The algebra 	
	\begin{equation*}
		{\Pi}^{(n,m)}:={\Pi}_{n+1}(A_m^n)
	\end{equation*}	
	is called the \emph{$(n+1)$-preprojective algebra of type $A_m$}. 	
	The algebras $A_m^n$ and ${\Pi}^{(n,m)}$ have combinatorial presentations as follows.
	
	\begin{definition} \cite[Definition 5.1]{IO1} \label{def.qns} Let $m,n \ge 1$.
		\begin{enumerate}
			\item Let
			\[
			Q_0^{(n,m)}:=
			\left\{x=(x_1,\ldots,x_{n+1})\in\mathbb Z_{\geq0}^{n+1}
			\ \middle|\ \sum_{i=1}^{n+1}x_i=m-1\right\}.
			\]
			Write $e_i$ for the $i$-th standard basis vector and put
			\[
			f_i=-e_i+e_{i+1}~(1\leq i\leq n),
			\quad f_{n+1}=-e_{n+1}+e_1.
			\]
			Whenever $x,x+f_i\in Q_0^{(n,m)}$, there is an arrow
			\[
			i:x \to x+f_i.
			\]
		
			\item  The $\mathbf{k}$-algebra $\widehat{\Lambda}^{(n,m)}$ is defined as the quiver algebra of $Q^{(n,m)}$ with the following relations:		
			for any $x\in Q_0^{(n,m)}$ and $i,j\in \{1, \ldots, n+1\}$ satisfying $x+f_i$, $x+f_i+f_j\in Q_0^{(n,m)}$,
			\[(x \xrightarrow{i} x+f_i \xrightarrow{j} x+f_i+f_j)
			=\left\{\begin{array}{cl}
				(x \xrightarrow{j} x+f_j \xrightarrow{i} x+f_i+f_j) & \text{if } x+f_j\in Q_0^{(n,m)},\\
				0&\text{otherwise.}
			\end{array}\right.\]
		\end{enumerate}
	\end{definition}
	
	Note that the algebra $\widehat{\Lambda}^{(n,m)}$ has $\binom{m+n-1}{n}$ primitive idempotents.
	
	\begin{theorem}[{See \cite[Theorems~5.6, 5.7 and Proposition~5.48]{IO1}}] \label{Preprojective Auslander cut}
	There are isomorphisms 
	\[
	{\Pi}^{(n,m)}\simeq \widehat{\Lambda}^{(n,m)} \quad \text{and} \quad 
	A_m^n \simeq \widehat{\Lambda}^{(n,m)}/(C_0)
	\] of $\mathbf{k}$-algebras, 	where $C_0$ is the set of all arrows of type $n+1$.
	\end{theorem}
	
	\begin{example} The quiver $Q^{(1,4)}$ is the following:
		\[
		\xymatrix{{ (3,0)} \ar@/^/[r]^{1} &{(2,1)} \ar@/^/[r]^{1}\ar@/^/[l]^{2} &{ (1,2)} \ar@/^/[r]^{1}\ar@/^/[l]^{2} &{ (0,3)}\ar@/^/[l]^{2}}
		\]
		The algebra $\widehat{\Lambda}^{(1,4)}$ is the classical preprojective algebra of type $A_4$, and moreover $A^1_4 \simeq \widehat{\Lambda}^{(1,4)}/(C_0)$.
	\end{example} 
	
	\begin{example}  \label{example} The quiver $Q^{(2,3)}$ is the following:
		$$\xymatrix@C=0.5cm@R0.5cm{
			&&020\ar[dr]^{2}\\
			&110\ar[dr]^{2}\ar[ur]^{1}&&011\ar[dr]^{2}\ar[ll]_{3}\\
			200\ar[ur]^{1}&&101\ar[ur]^{1}\ar[ll]_{3}&&002\ar[ll]_{3}
		}$$
		The algebra $\widehat{\Lambda}^{(2,3)}$ is the $3$-preprojective algebra of type $A_3$. Moreover, we have $A_3^2 \simeq \widehat{\Lambda}^{(2,3)}/(C_0)$.
	\end{example}
	
	\subsection{Algebras with radical square zero}
	In this subsection,  we collect some notions and properties of algebras with radical square zero. Denote
	by $J$  the Jacobson radical of an algebra. Throughout this subsection, we assume that $\Lambda$ is a basic \emph{radical square zero algebra}, that is, $J^{2}=0$.
	
	Associated with $\Lambda$ is the triangular matrix algebra	
	\begin{align*}
		\Sigma:=\begin{bmatrix}\Lambda/J&J\\[2pt]0&\Lambda/J\end{bmatrix},
	\end{align*}
	which will be used to study the representation theory of $\Lambda$. 
	A $\Sigma$-module is described by a \emph{triple} $(X^{\prime},X^{\prime\prime};\varphi)$,
	where $X^{\prime}$ and $X^{\prime\prime}$ are $(\Lambda/J)$-modules and
	$\varphi\colon X^{\prime}\otimes_{\Lambda/J}J\to X^{\prime\prime}$
	is a morphism in $\mod(\Lambda/J)$.
	A morphism
	$$f\colon(X^{\prime},X^{\prime\prime};\varphi)\to(Y^{\prime},Y^{\prime\prime};\psi)$$
	in $\mod\Sigma$ consists of a pair $(f',f'')=(f^{\prime}\colon X'\to Y', f''\colon X''\to Y'')$ of morphisms in
	$\mod(\Lambda/J)$ such that
		$f^{\prime\prime}\varphi=\psi\bigl(f^{\prime}\otimes_{\Lambda/J}J\bigr)$,
	that is, the diagram
	\begin{align*}
		\xymatrix{
			X^{\prime}\otimes_{\Lambda/J}J\ar[r]^-{\varphi}\ar[d]_{f'\otimes_{\Lambda/J} J}
			&X''\ar[d]^{f^{\prime\prime}}\\
			Y^{\prime}\otimes_{\Lambda/J}J\ar[r]^-{\psi}&Y^{\prime\prime}}
	\end{align*}
	commutes; see \cite[A.2.7]{ASS} and \cite[III.2]{ARS} for details.
	
	The structure of $\Sigma$ can be described by means of the
	separated quiver construction. Let $Q=(Q_{0},Q_{1})$ be a quiver. For each vertex $x\in Q_{0}$,  set
	$
	Q_{0}^{+}:=\{\,x^{+}\mid x\in Q_{0}\,\}
	$ and $
	Q_{0}^{-}:=\{\,x^{-}\mid x\in Q_{0}\,\}
	$.
		The \emph{separated quiver} of $Q$, denoted by
	$Q^{s}=(Q_{0}^{s},Q_{1}^{s})$, is defined by
	$Q_{0}^{s}:=Q_{0}^{+}\coprod Q_{0}^{-}$
	and
	\[
	Q_{1}^{s}:=
	\{\,x^{+}\longrightarrow y^{-}
	\mid x\longrightarrow y\text{ is an arrow of }Q\,\}.
	\]
		Thus, all arrows in $Q^{s}$ start at a vertex in $Q_{0}^{+}$ and end at
	a vertex in $Q_{0}^{-}$. In particular, every vertex of $Q^{s}$ is either a
	source or a sink. Notice also that the separated
	quiver $Q^s$ is not necessarily connected  even when  $Q$ is connected; see Example \ref{example:connected comp}.
	A full subquiver $Q'$ of $Q^{s}$ is called a \emph{single subquiver} if for any $x\in Q_{0}$, the vertex set $Q_{0}'$ contains at most one of $x^{+}$ or $x^{-}$.
	
	The following result explains the relationship between the triangular matrix
	algebra $\Sigma$ and the separated quiver of the  quiver for $\Lambda$.
	
	\begin{proposition}\cite[III.2.5]{ARS} \label{separated quiver algebra}
		Let $Q$ be the quiver of $\Lambda$. Then the algebra $\Sigma$ is isomorphic to the path algebra of $Q^{s}$. 	In particular, $\Sigma$ is a hereditary algebra with radical square zero.
	\end{proposition}
	
	We collect some results on the representation theory of algebras with radical square zero. We refer to \cite[Section X]{ARS} for more details.
	Define a functor $F\colon \mod \Lambda\to{} \mod\Sigma$ as follows.
	For any $\Lambda$-module $X$, let 
	\begin{align}
		F(X):=(X/XJ, XJ; \varphi_{X}),\notag
	\end{align} 
	where $\varphi_{X}\colon X/XJ\otimes_{\Lambda/J} J\to{} XJ$ is the morphism induced by the canonical multiplication map $X\otimes_{\Lambda} J\to{} XJ$ since $J^{2}=0$.
	For a morphism $g\colon X\to Y$ in $\mod\Lambda$, define 
	\begin{align}
		F(g):=(g^{\prime},g^{\prime\prime}), \notag
	\end{align} 
	where $g^{\prime}\colon X/XJ\to{} Y/YJ$ is induced by $g$ and $g^{\prime\prime}\colon XJ\to{} YJ$ is the restriction to $XJ$.
	\begin{proposition}{\cite[X.2.1, X.2.2, X.2.4, X.2.6]{ARS}}
		\label{2.6}
		The following hold:
		\begin{enumerate}
			\item The functor $F$ is full and induces an equivalence of categories 
			\[
			F\colon\underline{\mod}\,\Lambda\longrightarrow\underline{\mod}\,\Sigma
			\]			
			\item A $\Lambda$-module $X$ is indecomposable (resp. projective) if and only if $F(X)$ is an indecomposable (resp.  projective) $\Sigma$-module.
			\item  $\Lambda$ is representation-finite if and only if the separated quiver of the quiver of $\Lambda$ is a disjoint union of Dynkin quivers.
		\end{enumerate}
	\end{proposition}

	\begin{lemma}\label{lem:wild-separated}
		Let $Q^s$ be the separated
		quiver of  of the quiver of $\Lambda$. If a connected component of
		$Q^s$ is of wild hereditary type, then $\Lambda$ is wild.
	\end{lemma}
	
	\begin{proof}		
		Let $C$ be a connected component of $Q^s$ which is of wild hereditary
		type. Since $Q^s$ is a disjoint union of its connected components, we
		have
		\[
		\Sigma\simeq {\bf k}C\times\Sigma',
		\]
		for some finite dimensional algebra $\Sigma'$. In particular, ${\bf k}C$ is
		a quotient algebra of $\Sigma$. The functor
		$$
		\operatorname{mod} {\bf k}C\to \operatorname{mod}\Sigma
		$$
		is exact and fully faithful. Hence the wildness of ${\bf k}C$ implies that
		$\Sigma$ is wild.
		
		On the other hand, Proposition~\ref{2.6}(1) gives a stable equivalence
		\[
		\underline{\mod}\,\Lambda
		\simeq
		\underline{\mod}\,\Sigma.
		\]
		By \cite[Corollary 3.4]{Kr}, stable equivalence preserves wildness. Therefore
		$\Lambda$ is wild if and only if $\Sigma$ is wild. Since $\Sigma$ is
		wild, $\Lambda$ is wild.
	\end{proof}
	
	\subsection{Quivers with potential and $3$-preprojective algebras} 
	
	Throughout this subsection, we
	assume that $\Lambda$ is an algebra with ${\rm gl.dim}\,\Lambda\le 2$. For a quiver $Q$ and an arrow $a\in Q_1$, we denote by $s(a)\in Q_0$ its \emph{source} and $t(a)\in Q_0$ its \emph{target}, respectively. For a path $p=a_1a_2\dots a_l$ in $Q$, we also denote by $s(p)=s(a_1)$ and $t(p)=t(a_l)$ starting vertex and ending vertex, respectively.

	\begin{definition} \label{quiver with potential} (\cite{KV}; see \cite[Definition 2.1]{HI} for the completed version)
 For a presentation $\Lambda = {\bf k}Q/ {\langle r_1, \ldots, r_l \rangle}$ by a quiver $Q$ and a minimal set $\{r_1, \ldots, r_l\}$ of relations in $ {\bf k}Q$, 
		 define the \emph{quiver with potential} $(Q_\Lambda,W_\Lambda)$ by
			\begin{itemize}
				\item[$\bullet$] $Q_{\Lambda,0} = Q_0$,
				\item[$\bullet$] $Q_{\Lambda,1} = Q_1 \coprod C_\Lambda$ with $C_\Lambda:=\{\rho_i\colon t(r_i)\to s(r_i)\ |\ 1\le i\le l\}$,
				\item[$\bullet$] $W_\Lambda = \sum_{i=1}^l r_i\rho_i$.
			\end{itemize}			
	\end{definition}
			For a  cyclic path $p$ in $Q_\Lambda$ and $a\in Q_{\Lambda,1}$, let 
		$$\partial_a(p)=\sum_{p=uav}vu.$$	
		The \emph{Jacobian algebra} is defined by $\mathcal{P}(Q_\Lambda,W_\Lambda)={\bf k}Q_\Lambda/ \langle \partial_a W_\Lambda \;|\; a \in Q_{\Lambda,1} \rangle$.

	The following result gives the relationship between $(Q_\Lambda,W_\Lambda)$ and ${\Pi}_3(\Lambda)$.
	\begin{proposition}
		\cite[Proposition~2.2]{HI}, \cite[Theorem 6.10]{KV} The ${\bf k}$-algebra  ${\Pi}_3(\Lambda)$ and $\mathcal{P}(Q_\Lambda,W_\Lambda)$ 
		are isomorphic.
	\end{proposition}
	
	\begin{example}
		Let $\Lambda=A_3^2$ be the Auslander algebra of ${\bf k}A_3$. Then $\Lambda={\bf k} Q/I$, where $Q$ is depicted below and the ideal $I=(ac,bd-ce,ef)$.
		$$Q:\quad
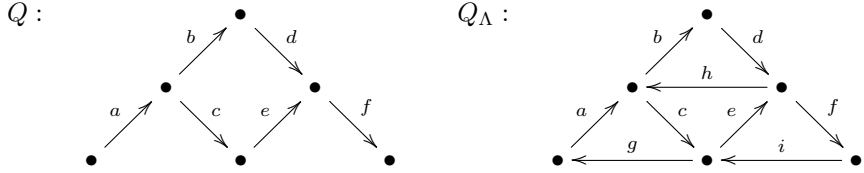
\xymatrix@C=0.6cm@R0.6cm{
			&&\bullet \ar[dr]^{d}\\
			&\bullet\ar[dr]^{c}\ar[ur]^{b}&&\bullet\ar[dr]^{f}\\
			\bullet\ar[ur]^{a}&&\bullet\ar[ur]^{e}&&\bullet
		}\quad \quad Q_\Lambda:\quad
		\xymatrix@C=0.6cm@R0.6cm{
		&&\bullet \ar[dr]^{d}\\
		&\bullet\ar[dr]^{c}\ar[ur]^{b}&&\bullet\ar[dr]^{f}\ar[ll]_{h}\\
		\bullet\ar[ur]^{a}&&\bullet\ar[ur]^{e}\ar[ll]_{g}&&\bullet\ar[ll]_{i}
		}
		$$
		By Definition \ref{quiver with potential}, $Q_{\Lambda}$ is depicted above, and $W_\Lambda=acg+bdh-ceh+efi$. Thus ${\Pi}_3(\Lambda)={\bf k}Q_{\Lambda}/(cg,dh,ga-eh,hb,fi-hc,ie,ac,bd-ce,ef)$.
	\end{example}
	
	\section{Proof of Theorem \ref{main thm}}
		In this section, we give the finite--tame--wild trichotomy of $\Pi^{(n,m)}$. For $n=1$, the algebra $\Pi^{(1,m)}$ coincides with the classical preprojective algebra of type $A_m$.  It is representation-finite precisely when
		$m\leq4$; in particular, $\Pi^{(1,1)},\dots, \Pi^{(1,4)}$ have $1,4,12$ and $40$ indecomposable modules, respectively.  The algebra
		$\Pi^{(1,5)}$ is tame, whereas $\Pi^{(1,m)}$ is wild for $m\geq6$; see
		\cite[Proposition~3.3 and Section 9.6]{GLS} and the references therein. 
		Therefore, it remains to consider $n\ge 2$ and $m\ge 1$.

		\subsection{The representation-finite cases}
				
		\begin{proposition} \label{representation-finite cases}
			For any $n\ge2$ and $m\le 3$,  ${\Pi}^{(n,m)}$ is representation-finite. Moreover, we have 
			\[
			|\Ind {\Pi}^{(n,m)}|=
			\begin{cases}
				1, & m=1, \\[3pt]
				2(n+1), & m=2, \\[3pt]
				\dfrac{(n+1)(n+2)(n+3)}{2}, & m=3, \\[8pt]
			\end{cases}
			\]
			where $|\Ind {\Pi}^{(n,m)}|$ denotes the number of isomorphism classes of indecomposable ${\Pi}^{(n,m)}$-modules.
		\end{proposition}	
		\begin{proof} We divide the proof into the three cases $m=1,2,3$.
			
			\emph{Case $1$}: For $m=1$, the quiver $Q^{(n,1)}$ has only one vertex and no arrows, so we have
			${\Pi}^{(n,1)}\simeq {\bf k}$.
			
			\emph{Case $2$}: For $m=2$, the algebra ${\Pi}^{(n,2)}$ is presented by the quiver 
			$$e_1 \xlongrightarrow{1} e_2 \xlongrightarrow{2} \dots \xlongrightarrow{n} e_{n+1}\xlongrightarrow{n+1} e_1$$ modulo the ideal generated by all paths of length two. Hence ${\Pi}^{(n,2)}$ is the
			 cyclic Nakayama algebra with radical square zero. Its indecomposable modules consist of $n+1$ simple modules $S_i$ and $n+1$ indecomposable projective modules $P_i$ for $1\leq i\leq n+1$. Its Auslander--Reiten quiver is depicted below.
			\begin{figure}[H]
			$$\xymatrix@C=0.5cm@R0.5cm{
				&P_{n+1}\ar[dr]&&P_n\ar[dr]&& \dots\ar[dr] &&P_{1}\ar[dr]\\
				S_1\ar[ur]\ar@{--}[rr]&&S_{n+1}\ar[ur]\ar@{--}[rr]&&\dots\ar[ur]\ar@{--}[rr]&&S_{2}\ar[ur]\ar@{--}[rr]&& S_1
				}$$
			\caption{The Auslander--Reiten quiver of ${\Pi}^{(n,2)}$}	\label{AR quiver for (n,2)}
			\end{figure}
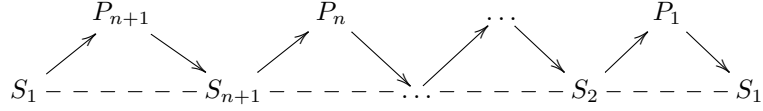\noindent 
		 Hence we have
			\begin{equation*}
				|\Ind{\Pi}^{(n,2)}|=2(n+1).
			\end{equation*}
						
			\emph{Case $3$}: For $m=3$,  we note that 
			\begin{equation*}
				\underline{\mod}\,\Pi^{(n,3)}
				\simeq
				\mathcal C_{A_{2}^{n+1}}^{n+1}.
				\label{eq:stable-equivalence}
			\end{equation*}
			by \cite[Theorem 5.47]{IO1}. Here,
			the algebra $A_2^{n+1}$ is the linear Nakayama algebra with radical square zero, which is presented by the quiver 									
			$$1 \longrightarrow 2 \longrightarrow \dots \longrightarrow {n+1}\longrightarrow n+2$$ 
			modulo the ideal generated by all paths of length two. Note that $A_2^{n+1}$ is derived equivalent to a hereditary algebra of Dynkin type $A_{n+2}$.  By (\ref{eq:dynkin-orbit}), 
			\begin{equation}\label{stable category of Pi^{(n,3)}}
				\underline{\mod}\,\Pi^{(n,3)}
				\simeq
				\DDD^{\bo}({\bf k}A_{n+2})/\tau^{-1}[n].				
			\end{equation}
		 A fundamental domain of the orbit category $\DDD^{\bo}({\bf k}A_{n+2})/\tau^{-1}[n]$ 
		 consists of $n$ copies of the set of indecomposable ${\bf k}A_{n+2}$-modules together with one copy of the indecomposable projective modules.
		Since $$|\Ind ({\bf k}A_{n+2})|=\frac{(n+2)(n+3)}2,$$ the number of stable indecomposable modules is
		\begin{align}
			\left|\Ind\underline{\mod}\,\Pi^{(n,3)}\right|
			&=n\frac{(n+2)(n+3)}2+(n+2)\notag\\
			&=\frac{(n+1)(n+2)^2}{2}.
			\label{eq:m3-stable-count}
		\end{align}
		There are
		$$
		\binom{n+2}{n}=\binom{n+2}{2}
		$$
		indecomposable projective $\Pi^{(n,3)}$-modules.  Adding them to
		\eqref{eq:m3-stable-count} gives
		\begin{equation*}
			|\Ind\Pi^{(n,3)}|
			=\frac{(n+1)(n+2)(n+3)}2.
			\label{eq:m3-total-count}
		\end{equation*} This completes the proof.
			\end{proof}
			

		\subsection{The tame and wild cases}
			Let
		\[
		B=B^{(n,m)}:=\Pi^{(n,m)}/\rad^2\Pi^{(n,m)}.
		\]
		
		\begin{proposition} \label{wild case} The algebra ${\Pi}^{(n,m)}$ is wild in each of the following cases:
			\begin{itemize}
				\item [(1)] $n=2$ and $m\ge 5$;
				\item [(2)] $n\ge 3$ and $m\ge 4$. 			
			\end{itemize}
		\end{proposition}
		\begin{proof} 		
			Since $B$ is a quotient algebra of $\Pi^{(n,m)}$, there exists a natural embedding $\mod B \to \mod \Pi^{(n,m)}$;
			hence it suffices to show that $B$ is wild.
			
			 Assume first that $n=2$ and $m\geq5$. Set $s=(m-5)e_1$.  In the separated
			quiver $Q^s_B$  of the quiver of $B$, the following six distinct vertices span a full subquiver whose underlying graph is the
			Euclidean diagram $\widetilde A_5$:
			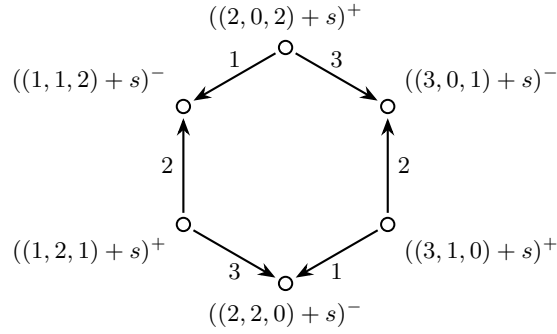
\begin{figure}[H]
				\begin{tikzpicture}[scale=1.2, 
					every node/.style={font=\small},
					vtx/.style={circle, draw, minimum size=5pt, inner sep=0pt, fill=white},
					arr/.style={->, >=Stealth, thick, shorten >=1.5pt, shorten <=1.5pt}]
					
					\node[vtx, label=90:{$((2,0,2)+s)^+$}]   (v1) at (90:1.3)  {};
					\node[vtx, label=30:{$((3,0,1)+s)^-$}]    (v2) at (30:1.3)  {};
					\node[vtx, label=-30:{$((3,1,0)+s)^+$}]   (v3) at (-30:1.3) {};
					\node[vtx, label=-90:{$((2,2,0)+s)^-$}]   (v4) at (-90:1.3) {};
					\node[vtx, label=-150:{$((1,2,1)+s)^+$}]  (v5) at (-150:1.3){};
					\node[vtx, label=150:{$((1,1,2)+s)^-$}]   (v6) at (150:1.3) {};
					
					\draw[arr] (v1) -- node[above] {$3$} (v2);
					\draw[arr] (v3) -- node[ right] {$2$} (v2);
					\draw[arr] (v3) -- node[below] {$1$}(v4);
					\draw[arr] (v5) --node[below] {$3$}  (v4);
					\draw[arr] (v5) -- node[left] {$2$} (v6);
					\draw[arr] (v1) -- node[above] {$1$} (v6);
				\end{tikzpicture}
				\caption{A full subquiver of $Q_B^s$ for $n=2$ and $m\ge 5$}	\label{A subquiver: wild 1}
			\end{figure}	\noindent 
			Moreover,  there is an arrow from $((1,2,1)+s)^+$  to the  vertex
			$((0,3,1)+s)^-$.  Thus  these
			seven vertices induce the graph obtained from $\widetilde A_5$ by attaching one leaf.
			
			  Now assume that $n\ge 3$ and $m\ge 4$. Put $s=(m-4)e_1$, and define the following vectors in $\mathbb{Z}^{n+1}$:
			\begin{align*}
				a&=e_1+e_3+e_{n+1}+s, & b&=2e_1+e_3+s,\\
				c&=2e_1+e_2+s,        & d&=e_1+2e_2+s,\\
				e&=2e_2+e_{n+1}+s,    & f&=e_2+e_3+e_{n+1}+s.
			\end{align*}
			Then the separated quiver $Q_B^s$ contains the full subquiver
				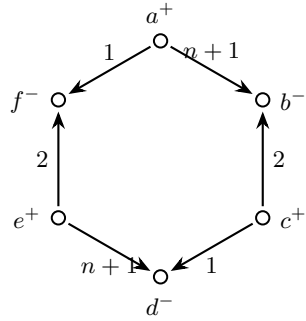
\begin{figure}[H]		\begin{tikzpicture}[scale=1.2, 
				every node/.style={font=\small},
				vertex/.style={circle, draw, minimum size=5pt, inner sep=0pt, fill=white},
				edge/.style={->, >=Stealth, thick, shorten >=1.5pt, shorten <=1.5pt}]
				
				\node[vertex, label=above:{$a^+$}] (a) at (90:1.3) {};
				\node[vertex, label=right:{$b^-$}] (b) at (30:1.3) {};
				\node[vertex, label=right:{$c^+$}] (c) at (-30:1.3) {};
				\node[vertex, label=below:{$d^-$}] (d) at (-90:1.3) {};
				\node[vertex, label=left:{$e^+$}] (e) at (-150:1.3) {};
				\node[vertex, label=left:{$f^-$}] (f) at (150:1.3) {};
				
				\draw[edge] (a) -- node[above] {$n+1$} (b);
				\draw[edge] (c) -- node[right] {$2$} (b);
				\draw[edge] (c) -- node[below] {$1$} (d);
				\draw[edge] (e) -- node[below] {$n+1$} (d);
				\draw[edge] (e) -- node[left] {$2$}(f);
				\draw[edge] (a) -- node[above] {$1$}(f);
			\end{tikzpicture}
		\caption{A full subquiver  of $Q_B^s$ for $n\ge 3$ and $m\ge 4$}	\label{A subquiver: wild 2}
	\end{figure}\noindent 	
			Define one more vector
			$
			g=e_1+e_4+e_{n+1}+s.
			$
			Note that there is an arrow from $a^+$ to $g^-$. Hence the seven vertices $a^+,b^-,c^+,d^-,e^+,f^-$ and $g^-$ span a full subquiver whose underlying graph is the Euclidean
			diagram $\widetilde A_5$ by attaching one leaf.

			In either case,  the Euclidean
			diagram $\widetilde A_5$ with an attached leaf is neither a Dynkin diagram nor a Euclidean diagram. Consequently, the connected component of the separated quiver containing this subquiver is of wild hereditary type, and thus $B$ is wild  by Lemma \ref{lem:wild-separated}.			
		\end{proof}

		\begin{proposition} \label{tame case}  ${\Pi}^{(2,4)}$ is tame representation-infinite of polynomial growth.
			
		\end{proposition}
		\begin{proof}
			We first show that $\Pi^{(2,4)}$ is representation-infinite.  In the separated
			quiver $Q_B^s$  of the quiver of $B=B^{(2,4)}$, the following six  vertices span a full subquiver whose underlying graph is the Euclidean diagram $\widetilde A_5$:
			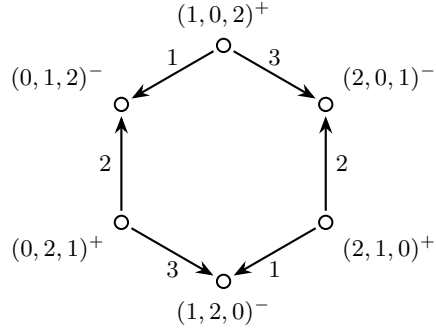
\begin{figure}[H]
			\begin{tikzpicture}[scale=1.2, 
				every node/.style={font=\small},
				vtx/.style={circle, draw, minimum size=5pt, inner sep=0pt, fill=white},
				arr/.style={->, >=Stealth, thick, shorten >=1.5pt, shorten <=1.5pt}]
				
				\node[vtx, label=90:{$(1,0,2)^+$}]   (v1) at (90:1.3)  {};
				\node[vtx, label=30:{$(2,0,1)^-$}]    (v2) at (30:1.3)  {};
				\node[vtx, label=-30:{$(2,1,0)^+$}]   (v3) at (-30:1.3) {};
				\node[vtx, label=-90:{$(1,2,0)^-$}]   (v4) at (-90:1.3) {};
				\node[vtx, label=-150:{$(0,2,1)^+$}]  (v5) at (-150:1.3){};
				\node[vtx, label=150:{$(0,1,2)^-$}]   (v6) at (150:1.3) {};
				
				\draw[arr] (v1) -- node[above] {$3$} (v2);
				\draw[arr] (v3) -- node[right] {$2$} (v2);
				\draw[arr] (v3) -- node[below] {$1$} (v4);
				\draw[arr] (v5) -- node[below] {$3$} (v4);
				\draw[arr] (v5) -- node[left] {$2$} (v6);
				\draw[arr] (v1) -- node[above] {$1$} (v6);
			\end{tikzpicture}	
		\caption{A full subquiver of $Q_B^s$ for $(n,m)=(2,4)$}\label{A subquiver: tame}	
	\end{figure}\noindent 
			  Then the radical square zero quotient $B^{(2,4)}$ is representation-infinite,
			and therefore ${\Pi}^{(2,4)}$ is as well.

			Recall that each algebra $\Lambda$ of global dimension at most two gives a quiver with potential  $(Q_{\Lambda},W_{\Lambda})$ whose Jacobian algebra is the $3$-preprojective algebra $\Pi_3(\Lambda)$ of $\Lambda$. Note that the algebra $A^2_4$ has global dimension two. Applied to $\Lambda=A^2_4$, this gives
			\begin{equation*}
			{\Pi}^{(2,4)}=\Pi_3(A_4^2)\simeq \mathcal{P}(Q,W),
				\label{eq:A42-Jacobian}
			\end{equation*}
			where $Q=Q^{(2,4)}$, and  \[
			W=
			\sum(\text{counter-clockwise cycles})-\sum(\text{clockwise cycles}).
			\]
			
			In Jasso's classification of self-injective cluster-tilted algebras of canonical type, 
			the Jacobian algebra of the quiver with potential $(Q,W)$ arises from the tubular cluster category of type $(2,3,6)$ 
			\cite[Theorem~1.3 and Figure~1.8]{Ja}. Indeed, by relabelling the 
			vertices, the first triangular quiver with potential in \cite[Figure~1.8]{Ja} coincides with
			$(Q,W)$.  Equivalently, there is a
			weighted projective line $\mathbb X$ of tubular type $(2,3,6)$ and
			a basic cluster-tilting object $T$ in its tubular cluster category
			$\mathcal C_{\mathbb X}=\DDD^{\bo}(\coh \X)/\tau^{-1}[1]$ such that
			\[
			\Pi_3(A_4^2)\simeq \End_{\mathcal C_{\mathbb X}}(T).
			\]
			It further follows from \cite[Theorem~1]{GGS} that the endomorphism algebra of each 
			cluster-tilting object in a tubular cluster category over an
			algebraically closed field is tame of polynomial growth.  Therefore $\Pi_3(A_4^2)$ is tame of polynomial growth.
		\end{proof}
		
		\begin{remark}
			The tame algebras $\Pi^{(1,5)}$ and $\Pi^{(2,4)}$  are both related
			to the tubular type $(2,3,6)$, but in different ways. For $\Pi^{(1,5)}$, its
			$\mathbb Z$-Galois covering is the repetitive algebra of a tubular
			algebra of type $(2,3,6)$ \cite{GLS}. 		
			In contrast,
			\(\Pi_3(A_4^2)\) is itself the Jacobian algebra of a
			 quiver with potential associated with a basic self-injective cluster-tilted algebras of 
			  tubular type $(2,3,6)$ \cite{Ja}.
			Thus the tubular structure appears through a Galois cover in the first case and directly through a
			Jacobian presentation in the second. In both cases, it accounts for
			tameness.
		\end{remark}

		\begin{proof}[Proof of Theorem~\ref{main thm}] The classical case $n=1$ is of finite type for $m\le4$, tame type for $m=5$,
			and wild type for $m\ge6$; see \cite[Proposition~3.3]{GLS} and the references therein. 		
			For $n\ge2$, Proposition~\ref{representation-finite cases} settles the cases $m\le3$ and gives the formulas, Proposition~\ref{tame case} treats $(n,m)=(2,4)$, and Proposition~\ref{wild case} proves wildness in all remaining cases.
		\end{proof}
		
%
		
		\section{Proofs of Theorems \ref{main thm 2} and \ref{main thm 3}}
		In this section, we give  complete proofs of Theorems \ref{main thm 2} and \ref{main thm 3}.
		 Let $\Lambda$	be an algebra.	
		Recall that a
		$\Lambda$-module $M$ is $\tau$-rigid if
		\[
		\Hom_\Lambda(M,\tau M)=0,
		\] where $\tau$ denotes the Auslander-Reiten translation of $\Lambda$.
		The algebra $\Lambda$ is called \emph{$\tau$-tilting finite} if there
		are only finitely many isomorphism classes of indecomposable
		$\tau$-rigid $\Lambda$-modules. Equivalently, $\Lambda$ has only
		finitely many isomorphism classes of basic support $\tau$-tilting
		modules; see \cite{DIJ}.
	
		The following theorem gives an analogue of  Proposition \ref{2.6}(3)  for $\tau$-tilting finiteness. It is a key tool for proving Theorem~\ref{main thm 2}.
		\begin{theorem} \label{Ad theorem} \cite[Theorem 3.1]{Ad}
			Let $\Lambda$ be an algebra with radical square zero.
			Then the following are equivalent:
			\begin{itemize}
				\item[(1)] $\Lambda$ is $\tau$-tilting finite.
				\item[(2)] Every single subquiver of the separated quiver for $\Lambda$ is a disjoint union of Dynkin quivers. 
			\end{itemize}
		\end{theorem}
		
		Now we are ready to prove Theorem~\ref{main thm 2}.
		
		\begin{proof}[Proof of Theorem~\ref{main thm 2}] 
			Suppose first that $n=1$.  Then ${\Pi}^{(1,m)}$ is the classical
			preprojective algebra of type $A_m$. By \cite[Theorem~2.21]{Mi}, there is a
			bijection
			\[
			\mathrm{s}\tau\text{-}\mathrm{tilt}\,{\Pi}^{(1,m)}
			\longleftrightarrow W(A_m)
			\]
			between basic support $\tau$-tilting modules and the Weyl group of $A_m$.  Since
			 $W(A_m)$ is finite, this proves $\tau$-tilting finiteness for
			each $m$. 
			
			Suppose next that $n\ge 2$ and $m\le 3$. 
			By Theorem~\ref{main thm}, we have that ${\Pi}^{(n,m)}$ is representation-finite, 
			hence $\tau$-tilting finite. 
			
			It remains to consider $n\ge 2$ and $m\geq4$. Set
			\[
			B^{(n,m)}:=\Pi^{(n,m)}/\rad^2\Pi^{(n,m)}.
			\]
			The separated quiver  of the quiver of $B^{(n,m)}$ contains a single subquiver whose underlying graph is the Euclidean diagram $\widetilde A_5$; see Figure \ref{A subquiver: tame} for  $(n,m)=(2,4)$, and  Figure \ref{A subquiver: wild 1}  for $n=2$ and $m\ge 5$, and  Figure \ref{A subquiver: wild 2} for $n\ge 3$ and $m\ge 4$.  By Theorem \ref{Ad theorem},  $B^{(n,m)}$ is not $\tau$-tilting finite. It follows from \cite[Corollary 1.9]{DIRRT} that
			the class of \(\tau\)-tilting finite algebras is closed under taking quotient algebras, and thus $\Pi^{(n,m)}$ is not $\tau$-tilting finite; otherwise  $B^{(n,m)}$ would be $\tau$-tilting finite, a contradiction.			
		\end{proof}

		\begin{lemma}\label{lem:m3-radical}
			Let $n\geq2$, $\Lambda=\Pi^{(n,3)}$ and $J=\rad \Lambda$. For $x=(x_1,\ldots,x_{n+1})\in Q_0^{(n,3)}$, put
			$
			\rho(x):=(x_{n+1},x_1,\ldots,x_{n}).
			$
			Then the following hold.
			\begin{enumerate}
				\item $J^3=0$, and $e_xJ^2$ is one-dimensional with terminating vertex
				$\rho(x)$. In particular, we have
				\[
				\soc(e_x\Lambda)=e_xJ^2\simeq S_{\rho(x)}.
				\]
				\item Let $B=\Lambda/J^2$ and $Q^s$ be the
				separated quiver of the quiver $Q$ of $B$.
				Then $Q^s$ has $n+1$ connected components, and each of them  has the underlying graph $\mathbb{A}_{n+2}$. Moreover,
				for each $x\in Q_0$, the vertices $x^+$ and $x^-$ belong to distinct connected components.
			\end{enumerate}
		\end{lemma}
		\begin{proof} (1) 
			Let $p$ be a nonzero path starting at $x$, and $a_i$ denote the number of
			arrows of type $i$ occurring in \(p\). If \(ij\)
			is a subpath of \(p\), then Definition~\ref{def.qns} gives either
			\(ij=ji\), provided that \(ji\) exists, or \(ij=0\). The latter is
			impossible since \(p\ne0\). Thus adjacent arrows in \(p\) may be
			interchanged. Moving all arrows of type $i$ to the beginning of $p$ gives $a_i\leq x_i$ for $1\le i\le n+1$. Since $\sum_{i=1}^{n+1} x_i=2$, $p$ has length at most two, and therefore $J^3=0$.
			
			Now assume that $p$ has length two. Since $\sum_{i=1}^{n+1} a_i=2= \sum_{i=1}^{n+1} x_i$ and $a_i\le x_i$ for all $i$, it follows that  $a_i=x_i$ for all $i$.
			The terminating vertex of $p$ is $$t(p)=x+\sum_{i=1}^{n+1}x_if_i=\rho(x).$$ 
			There is a nonzero path of length two starting at $x$. Indeed, if $x_i=2$, the path $x \xrightarrow{i} x+f_i \xrightarrow{i} x+2f_i$  is nonzero; if $x_i=x_j=1$ with $i\neq j$, then the paths $ij$ and $ji$ are  nonzero and are identified by a commutativity relation. Successive interchanges identify any two nonzero length-two paths starting at $x$.
			 Hence $e_xJ^2$ is one-dimensional. Since $J^3=0$, one has $e_xJ^2\subset \soc(e_x\Lambda)$.
			Since $\Lambda$ is
			self-injective, the indecomposable projective module $e_x\Lambda$ has simple			
			socle, which is necessarily $e_xJ^2$.
			
			(2) Put $I=\Z/(n+1)\Z$. We use cyclic indices in $I$.  Note that the quiver of $B$ is
			$Q=Q^{(n,3)}$. Since \(\sum_{i=1}^{n+1}x_i=2\), the vertex $x$ records the positions
			of two units among the $n+1$ coordinates. Then
			each vertex \(x\) is uniquely represented by a two-element
			multiset \(\langle a,b\rangle\) of elements of \(I\). 	
			Define a map $\ell\colon Q_0\to I$ given by
			\[
			\ell(x)=\ell(\langle a,b\rangle):={a}+{b}, 
			\]  
			Then, for each arrow $x\to y$ in $Q^{(n,3)}_1$, we have $\ell(y)=\ell(x)+{1}$. 
			Moreover, we define a map $\theta \colon Q^{s}_0 \to I$ given by 
			$$ \theta(x^{+})=\ell(x) \quad \text{and} \quad \theta(x^{-})=\ell(x)-{1}$$  for any $x\in Q_0$.
			Then, for each arrow $x^{+}\to y^{-}$ in $Q^{s}$,
			$\theta(y^{-})=\ell(y)-{1}=\ell(x)=\theta(x^{+}),$ so $\theta$ is constant on each connected component of $Q^{s}$.
			
			Fix $ c\in I$, and let $\Gamma_{ c}$ be the full subquiver of $Q^s$ on the fiber
			$\theta^{-1}( c)$. Its  vertices may be written as
			$$p_{ a}=\langle a,c-a\rangle^+ \quad\text{and} \quad q_{ a}=\langle a, c+ 1- a\rangle^-\quad( a\in I)$$
			subject to the identifications
			$$p_{ a}=p_{  c-  a}\quad \text{and}\quad q_{ a}=q_{ c+ 1-  a}\quad ( a\in I).$$
			For $ d\in I$, put
			$
			r_{ d}=\#\{ a\in I\mid  {2a}=d\}.
			$
			Hence we have
			\[
			|(\Gamma_{c})_0|
			=\frac{n+1+r_c}{2}+\frac{n+1+r_{c+ 1}}{2}.
			\]
			If $n+1$ is odd, then the map $a\mapsto 2a$ on $I$ is bijective, and thus $r_c=r_{c+ 1}=1$. 
			If $n+1$ is even, exactly
			one of $c$ and $c+ 1$ lies in $2I$;  for that element the equation $2a=d$ has two solutions, and for the other it has none.
			 In either case, we have $r_c+r_{c+ 1}=2$, and thus
			\[
			|(\Gamma_c)_0|=n+2.
			\]
			
			For each \(a\in I\), there is precisely one arrow of type $a$ in
			\(\Gamma_c\), namely
			$$\alpha_a\colon p_a\longrightarrow q_{a+ 1}.$$ 
			Indeed, 
			let $\alpha\colon x^+ \to y^-$ be an arrow of type \(a\) in \(\Gamma_c\).
			Then $x^+$ contains $a$ and \(\theta(x^+)=c\);
			hence 
			$x^+=\langle a,c-a\rangle^{+}=p_a$.
			The arrow of type $a$ with source
			$p_a=\langle a,c-a\rangle ^{+}$  gives the target
			$y^-=\langle a+ 1,c-a\rangle^{-}=q_{a+ 1}.$
		 Thus
			\(a\mapsto\alpha_a\) gives a bijection
			$
			I\xrightarrow{\sim}(\Gamma_c)_1,
			$
			and we have $$|(\Gamma_c)_1|=n+1.$$
			
			The identities $p_{c-a-1}=p_{a+1}$ and $q_{c-a}=q_{a+1}$ give
			$$p_a\xrightarrow{\alpha_{a}}q_{a+1}\xleftarrow{\alpha_{c-a-1}}p_{a+1}.$$
			Thus $p_a$ and $p_{a+1}$ lie in the same connected component for each $a\in I$, so all vertices $p_a$, ${a\in I}$ are connected. Each $q_a$ is the target of $\alpha_{a-1}$ and hence $\Gamma_c$ is connected.
			
			Let $\widetilde{\Gamma}_c$ be the underlying graph of $\Gamma_c$. Since $\Gamma_c$ is connected and has $n+2$ vertices and $n+1$ arrows, its underlying graph $\widetilde{\Gamma}_c$ is a tree. In particular,
			it has no multiple edges.  Each vertex of $\widetilde{\Gamma}_c$ has degree at
			most two. Indeed, for $p_a$, only arrows of types $a$ and $c-a$ may occur, and
			for $q_a$, only arrows of types $a-{1}$ and $c-a$ may occur.
			Consequently, $\widetilde{\Gamma}_c$ has the underlying graph $\mathbb{A}_{n+2}$.

			Each fiber of $\theta$ is connected, whereas $\theta$ is constant on
			connected components. Hence the $\Gamma_c$, $c\in I$, are precisely
			the connected components of $Q^{\mathrm s}$, and their number is
			$|I|=n+1$. Finally,
			\[
			\theta(x^+)=\ell(x)
			\ne \ell(x)- 1=\theta(x^-),
			\]
			since $n+1\geq3$. Thus $x^+$ and $x^-$ belong to distinct connected
			components.			
		\end{proof}

			\begin{example}\label{example:connected comp} 			 						
				Let $B=\Pi^{(2,3)}/\rad^2\Pi^{(2,3)}$. The quiver $Q$ of $B$ is $Q^{(2,3)}$; see Example \ref{example}. Its separated quiver  $Q^s$ has the three connected components: 
				\[
				\xymatrix@C=1.2cm{
					020^-
					&
					110^+\ar[l]_{1}\ar[r]^{2}
					&
					101^-
					&
					002^+\ar[l]_{3}
				},
				\]
				\[
				\xymatrix@C=1.2cm{
					020^+\ar[r]^{2}
					&
					011^-
					&
					101^+\ar[l]_{1}\ar[r]^{3}
					&
					200^-
				},
				\]
				\[				
				\xymatrix@C=1.2cm{
					200^+\ar[r]^{1}
					&
					110^-
					&
					011^+\ar[l]_{3}\ar[r]^{2}
					&
					002^-
				}.
				\]
				Each of them  has underlying graph $\mathbb{A}_{4}$. 
				Moreover,
				for each $x\in Q_0$, the vertices $x^+$ and $x^-$ belong to distinct connected components.			
			\end{example}
		
			Recall that for an algebra $\Lambda$,  $M\in \mod \Lambda$ is  called a \emph{brick} if $\End_\Lambda(M)\simeq {\mathbf{k}}$.

			\begin{proposition}\label{prop:all-tau-rigid}
				For any $n\geq2$ and $m\leq3$,  each indecomposable
				$\Pi^{(n,m)}$-module is a brick and $\tau$-rigid.
			\end{proposition}
			
			\begin{proof} We first prove the brick assertion.
				For \(m=1\),  one has $\Pi^{(n,1)}\simeq{\mathbf k}$, and the assertion is immediate.
				
				For $m=2$,
				the algebra $\Pi^{(n,2)}$ is a cyclic Nakayama algebra with radical square zero.
				Its indecomposable modules are the simple modules and the indecomposable
				projective modules, and each of them has endomorphism algebra \({\mathbf k}\); see Figure \ref{AR quiver for (n,2)}.
				
				For \(m=3\), let
				$\Lambda=\Pi^{(n,3)}$ and $J=\rad\Lambda$.
				First we consider an indecomposable projective module
				\(P_x=e_x\Lambda\). By Lemma \ref{lem:m3-radical},  $J^3=0$, and every nonzero path of length two 
				starting at $x$ terminates at $\rho(x)$. Moreover, $\rho(x)\neq x$; otherwise all coordinates of $x$  would be equal, which is impossible because $n+1\ge 3$ and $\sum_{i=1}^{n+1} x_i=2$.  Since the quiver has no loops, there is no nontrivial nonzero path starting and ending at $x$.
				 Hence
				\[
				\End_{\Lambda}(P_x)
				\simeq e_x\Lambda e_x
				\simeq{\mathbf k}.
				\]
				
				Now let \(M\) be an indecomposable nonprojective \(\Lambda\)-module. We
				claim that
				$
				MJ^2=0.
				$
				Suppose that $MJ^2\neq 0$.  There exist a vertex $x\in Q_0^{(n,3)}$ and 
				\(u\in Me_x\) such that \(uJ^2\neq0\). This induces a morphism
				\[
				\varphi_u\colon P_x=e_x\Lambda\to M,
				\quad
				a\mapsto ua
				\] in $\mod \Lambda$. 
				By Lemma \ref{lem:m3-radical}, 
				$\soc P_x=e_xJ^2$ is simple, and the restriction of $\varphi_u$ to this socle is nonzero. Thus $\ker\varphi_u\cap\soc P_x=0$.  Since the socle of $P_x$ is essential,
				we have \(\ker\varphi_u=0\), and thus
				\(\varphi_u\) is injective. Since $\Lambda$ is self-injective, $P_x$ is injective; hence \(\varphi_u\) splits. Consequently, \(P_x\) is a direct summand of \(M\), contradicting the
				assumption that \(M\) is indecomposable and nonprojective. Hence
				\(MJ^2=0\).

				Set
				$
				B=\Lambda/J^2$,
				$
				\overline J=J/J^2=\rad B$ and
				\begin{align*}
					\Sigma=\begin{bmatrix}B/\overline J&\overline J\\[2pt]0&B/\overline J\end{bmatrix}.
				\end{align*} 
				 Since $MJ^2=0$,
				 $M$ is naturally an indecomposable $B$-module. Consider the functor (see the definition in the paragraph before Proposition \ref{2.6})
				\[
				F\colon\mod B\to\mod\Sigma,
				\quad
				M\mapsto (M/M\overline J,M\overline J;\varphi_M);
				\]
				By Proposition \ref{2.6}(2), \(F(M)\) is indecomposable. Note that the quiver of $B$ is
				$Q=Q^{(n,3)}$.  By	Lemma~\ref{lem:m3-radical}(2),  each connected component of the separated quiver $Q^s$
				of $Q$ has underlying graph \(\mathbb A_{n+2}\). Then \(F(M)\) is an
				indecomposable representation of a quiver of type \(\mathbb A_{n+2}\), and 
				 thus
				\[
				\End_{\Sigma}(F(M))\simeq{\mathbf k}.
				\]
				
				Next we show
				that the natural map
				\[
				\Phi_M\colon\End_B(M)\to\End_{\Sigma}(F(M)), \quad f\mapsto F(f)
				\]
				is injective. Let $f\in\End_B(M)$ such that $F(f)=0$. By the
				definition of \(F\), the induced map on \(M/M\overline J\) and the
				restriction of \(f\) to \(M\overline J\) are both zero. Hence
				$f(M)\subseteq M\overline J$ and $ f(M\overline J)=0$,
				so \(f\) induces a \(B/\overline J\)-module homomorphism
				\[
				\widetilde f\colon M/M\overline J\to M\overline J,
				\quad
				m+M\overline J\mapsto f(m).
				\]			
				Since $F(M)$ is indecomposable, its support is contained in only one
				connected component of $Q^s$. If there exists a vertex \(x\) such that 
				$$(M/M\overline J)e_x\neq0 \quad\text{and}\quad (M\overline J)e_x\neq0,$$ 
				then  both \(x^+\) and \(x^-\) belong to the support of \(F(M)\), and
				hence they lie in the same connected component of \(Q^s\). This contradicts
				Lemma~\ref{lem:m3-radical}(2). Therefore the two semisimple \(B/\overline J\)-modules
				\(M/M\overline J\) and \(M\overline J\) have disjoint supports, and
				so
				\[
				\Hom_{B/\overline J}
				(M/M\overline J,M\overline J)=0.
				\]
				This shows that \(\widetilde f=0\), and so \(f=0\). Thus $\Phi_M$ is
				 injective, and moreover $\Phi_M$ is
				 an isomorphism since $F$ is full by Proposition \ref{2.6}(1).
				 Hence we obtain
				\[
				\End_{\Lambda}(M)
				=\End_B(M)
				\simeq{\mathbf k}.
				\] 
				
			We have proved that each indecomposable module is a brick. By
			Theorem~\ref{main thm 2}, \(\Pi^{(n,m)}\) is \(\tau\)-tilting finite for
			\(n\geq 2\) and \(m\leq 3\). Hence, by \cite[Theorem 4.2]{DIJ}, there is a bijection between
			the isomorphism classes of indecomposable \(\tau\)-rigid modules and
			bricks. Thus we have
			\[
			\bigl|\itrig \Pi^{(n,m)}\bigr|
			=
			\bigl|\operatorname{brick}\Pi^{(n,m)}\bigr|
			=
			\bigl|\Ind \Pi^{(n,m)}\bigr|,
			\] where for an algebra $\Lambda$, $|\itrig \Lambda|$ and $|\operatorname{brick} \Lambda|$ denote the number of  indecomposable \(\tau\)-rigid modules and bricks, respectively. This shows that each indecomposable \(\Pi^{(n,m)}\)-module is \(\tau\)-rigid.
			\end{proof}
		
		Now we are ready to prove Theorem~\ref{main thm 3}.
		
		\begin{proof}[Proof of Theorem~\ref{main thm 3}] (1) $\Leftrightarrow$ (2) $\Leftrightarrow$ (3) Immediate from Theorems \ref{main thm} and \ref{main thm 2}.
									
			(1) $\Rightarrow$ (4) This follows directly from Proposition \ref{prop:all-tau-rigid}.
			
			(4) $\Rightarrow$ (1) Suppose that $m\ge 4$. By
			Theorem~\ref{main thm}, $\Pi^{(n,m)}$ is representation-infinite. By assumption, each indecomposable  $\Pi^{(n,m)}$-module is $\tau$-rigid. It follows from \cite[Theorem 1.1]{MP} that $\Pi^{(n,m)}$ is locally representation-directed, and thus it is representation-finite by \cite[Proposition 4.1]{MP}, a contradiction.			
		\end{proof}

			\vskip 5pt
		\noindent {\scriptsize   \noindent  Weikang Weng\\
			School of Mathematical Sciences, \\
			Xiamen University, Xiamen, 361005, Fujian, PR China.\\
			E-mail: 
			wkweng@stu.xmu.edu.cn\\ }
		\vskip 3pt
		
	\end{document}